\documentclass[12pt]{article}

\usepackage{fullpage}
\usepackage{amssymb}
\usepackage{amsmath}
\usepackage{amsthm}
\usepackage{stmaryrd}
\usepackage{xypic}
\usepackage{setspace}
\usepackage{array}
\usepackage{rotating}
\usepackage{url}

\newtheorem{theorem}{Theorem}[section]

\newtheorem{proposition}[theorem]{Proposition}
\newtheorem{corollary}[theorem]{Corollary}

\newtheorem{problem}{Problem}

\newcommand{\tref}[1]{Theorem~\textup{\ref{#1}}}
\newcommand{\pref}[1]{Proposition~\textup{\ref{#1}}}
\newcommand{\cref}[1]{Corollary~\textup{\ref{#1}}}

\newcommand{\s}{\sigma}
\newcommand{\G}{\Gamma}

\newcommand{\comment}[1]{}

\newcommand{\ds}{\displaystyle}

\DeclareMathOperator{\rank}{rank}

\newcommand{\calK}{\mathcal K}
\newcommand{\calL}{\mathcal L}
\newcommand{\calM}{\mathcal M}
\newcommand{\calP}{\mathcal P}
\newcommand{\calQ}{\mathcal Q}

\newcommand{\calR}{\mathcal R}

\begin{document}

\title{Non-flat regular polytopes and restrictions on chiral polytopes}

\author{Gabe Cunningham\\
Department of Mathematics\\
University of Massachusetts Boston\\
Boston, Massachusetts 
}

\date{ \today }
\maketitle

\begin{abstract}
\vskip.1in
\medskip
\noindent

An abstract polytope is \emph{flat} if every facet is incident on every vertex. In this paper, we
prove that no chiral polytope has flat finite regular facets and finite regular vertex-figures.
We then determine the three smallest non-flat regular polytopes in each rank, and
use this to show that for $n \geq 8$, a chiral $n$-polytope has at least $48(n-2)(n-2)!$ flags.


\comment{
	20B25: Finite automorphism groups of algebraic, geometric, or combinatorial structures.
	51M20: Polyhedra and polytopes; regular figures, division of spaces
	52B05: Combinatorial properties of polyhedra (number of faces etc.)
	52B15: Symmetry properties of polytopes
}

\end{abstract}

\section{Introduction}

	In many applications involving convex polytopes, what is most important is the \emph{combinatorial type}
	of the polytope: how many faces are there in each dimension, and which faces are incident. An
	\emph{abstract polytope} is essentially a partially ordered set that resembles the incidence
	relation for a convex polytope or a tiling of a surface or space.
	
	\emph{Regular (abstract) polytopes} are those that are maximally symmetric.
	The symmetry group of a regular polytope can be written in a standard form,
	and in fact the polytope can be recovered from a group presentation in this form. This means that
	many questions about regular polytopes can be translated to questions in group theory. Furthermore,
	this makes it possible to collect a large amount of data about regular polytopes, using standard
	group theory algorithms.

	In \cite{conder-atlas}, Conder catalogs the regular polytopes with up to 4000 flags (where a flag is
	a maximal chain of incidences). He excludes \emph{degenerate} polytopes, such as the digon, which
	consists of two edges and two vertices, with both edges edges incident on both vertices.
	However, many of the listed polytopes possess the minor degeneracy of being
	\emph{flat}, meaning that every facet is incident to every vertex. In ranks 4 and higher, more than 95\%
	of the listed polytopes are flat. See Table 1 for a summary of the counts.

	\begin{table}[t]
	\begin{center}
	\caption{Number of non-degenerate polytopes, up to duality, with up to 4000 flags}
	\begin{tabular}{c r r}
	\hline
	Rank & Flat & Non-flat \\ \hline
	3 &  2292 & 8186 \\
	4 &  7530 &  353 \\
	5 &  1561 &    8 \\
	6 &    52 &    0 \\ \hline
	\end{tabular}
	\end{center}
	\label{flat-data}
	\end{table}

	\emph{Chiral polytopes} are those that are fully symmetric under combinatorial rotations, but without
	mirror symmetry. Each chiral polytope is built out of regular and chiral polytopes of one dimension
	lower. One of the fundamental problems in the study of chiral polytopes is the \emph{amalgamation problem}:
	which polytopes can be assembled together to form a chiral polytope? In \tref{chiral-no-flat}, we will
	prove that no finite chiral polytope is built from flat regular polytopes that are arranged in a regular way
	around each vertex. Using this, we are able to describe several other restrictions on the structure
	of chiral polytopes in Section~\ref{restrictions}.

	Another important problem is the determination of the smallest chiral polytopes in each rank.
	Section~\ref{smallest-chiral-polys} describes what is currently known. In Section~\ref{non-flat-size}, we
	determine the smallest non-flat regular polytopes in each rank. Using this and \tref{chiral-no-flat},
	we prove in \tref{smallest-chiral-n} that, for $n \geq 8$, a chiral $n$-polytope has at least $48(n-2)(n-2)!$ flags.
	
\section{Background}

	\subsection{Abstract polytopes}

	Let us start with the definition of an abstract polytope, taken from \cite[Sec. 2A]{arp}.
	Consider a partially-ordered set $\calP$ with a unique minimal element and a unique maximal element.
	Suppose that the maximal chains of $\calP$ all have the same length. Then we can endow $\calP$ with a
	rank function, where the minimal element has rank $-1$, the elements that directly cover
	it have rank 0, and so on. Then $\calP$ is an \emph{(abstract) $n$-polytope} or \emph{polytope
	of rank $n$} if its maximal element has rank $n$ and if $\calP$ also satisfies the following conditions.
	\begin{enumerate}
	\item (Diamond condition): Whenever $F < G$ and $\rank G - \rank F = 2$, there are
	exactly two elements $H$ with $\rank H = \rank F + 1$ such that $F < H < G$.
	\item (Strong connectivity): Suppose $F < G$ and $\rank G - \rank F \geq 3$. If $F < H < G$
	and $F < H' < G$, then there is a chain
	\[ H = H_0 \leq H_1 \geq H_2 \leq H_3 \geq H_4 \leq \cdots \geq H_k = H' \]
	such that $F < H_i < G$ for each $i$.
	\end{enumerate}
	
	For example, the face-lattice of any convex $n$-polytope is an (abstract) $n$-polytope. In analogy with convex
	polytopes, we call the elements of $\calP$ \emph{faces}, and a face of rank $k$ is a \emph{$k$-face}.
	The faces of rank 0, 1, and $n-1$ are called \emph{vertices}, \emph{edges}, and \emph{facets}, respectively.
	The maximal chains of $\calP$ are called \emph{flags}, and two flags that differ in only a single
	element are said to be \emph{adjacent}.

	If $F < G$ are faces of an $n$-polytope, then the \emph{section $G/F$} consists of all faces $H$
	such that $F \leq H \leq G$. If $F$ is a facet of the $n$-polytope $\calP$ and $F_{-1}$ is the minimal face, 
	then the section $F/F_{-1}$ is an $(n-1)$-polytope. Usually, when we speak of a facet of $\calP$, we have 
	in mind this polytope, rather than just an element of rank $n-1$. If $v$ is a vertex of $\calP$ and $F_n$ is 
	the maximal face, then the section $F_n/v$ is also an $(n-1)$-polytope, called the \emph{vertex-figure at $v$}. 
	Given both a facet $F$ and a vertex $v$, the section $F/v$ is an $(n-2)$-polytope, called a \emph{medial section} 
	of $\calP$; it is both a vertex-figure of the facet $F$ and a facet of the vertex-figure at $v$.

	For each integer $p \geq 2$, there is a unique 2-polytope with $p$ vertices, denoted by $\{p\}$. When $p \geq 3$, this
	is simply the face-lattice of a $p$-gon; the case $p = 2$ yields the \emph{digon}, which has two
	edges, each of which is incident on two vertices. There is also a unique infinite 2-polytope, 
	denoted $\{\infty\}$, which is the face-lattice of the tiling of the real line by unit line segments.

	Given faces $F < G$ where $\rank G = \rank F + 3$, the section $G/F$ is
	a 2-polytope with some number $p(F, G)$ of vertices. If $\calP$ has the property that $p(F, G)$ depends only on
	the rank of $F$ and $G$ (rather than on the particular choice of faces in those ranks), then
	we say that $\calP$ is \emph{equivelar}. In this case, there are numbers $p_1, \ldots, p_{n-1}$
	such that, given any $(i-2)$-face $F$ and $(i+1)$-face $G$ with $F < G$, the section $G/F$ is
	the polytope $\{p_i\}$. We then say that $\calP$ has \emph{Schl\"afli symbol} (or \emph{type})
	$\{p_1, \ldots, p_{n-1}\}$.

	If $\calP$ and $\calQ$ are both $n$-polytopes, then a \emph{covering} $\pi: \calQ \to \calP$ is 
	a function that preserves the partial order, the rank of each face, and with the property that if
	two flags of $\calQ$ are adjacent, then their images under $\pi$ are also adjacent.
	(Such a function is automatically surjective.) 	We say then that $\calQ$ \emph{covers} $\calP$. 
	An \emph{isomorphism} of $n$-polytopes is a bijection that preserves rank and the partial order.
	
	If the facets of a polytope $\calP$ are all isomorphic to $\calK$, and the vertex-figures
	are all isomorphic to $\calL$, then we say that $\calP$ is of type $\{\calK, \calL\}$.
	If $\calP$ is of type $\{\calK, \calL\}$ and it covers all other polytopes of type $\{\calK, \calL\}$,
	then we call $\calP$ the \emph{universal} polytope of type $\{\calK, \calL\}$, and often denote
	it simply by $\{\calK, \calL\}$. This notation is naturally recursive, so that one may refer
	to a polytope such as $\{\{\calK, \calL\}, \{\calL, \calM\}\}$.
	
	The \emph{dual} of $\calP$, denoted $\calP^{*}$, is the polytope with the same underlying set as $\calP$ but
	with the partial order reversed. If $\calP$ has Schl\"afli symbol $\{p_1, \ldots, p_{n-1}\}$,
	then $\calP^{*}$ has Schl\"afli symbol $\{p_{n-1}, \ldots, p_1\}$, and if $\calP$ is of type
	$\{\calK, \calL\}$, then $\calP^{*}$ is of type $\{\calL^{*}, \calK^{*}\}$.
	When we say that something is true of $\calP$ \emph{up to duality}, we mean that it is either
	true of $\calP$ or of $\calP^{*}$.

	\subsection{Regular and chiral polytopes}
	
	The automorphism group of $\calP$, denoted $\G(\calP)$, consists of the isomorphisms from
	$\calP$ to itself. This group acts freely on the flags of $\calP$. A polytope is \emph{regular} if
	$\G(\calP)$ acts transitively on the flags. 
	Regular polytopes are in one-to-one correspondence with \emph{string C-groups},
	which we describe presently. Suppose that $\G = \langle \rho_0, \ldots, \rho_{n-1} \rangle$,
	where the generators $\rho_i$ satisfy at least the relations
	\begin{equation}
	\label{eq:involutions}
	\rho_i^2 = 1, \textrm{ for $0 \leq i \leq n-1$},
	\end{equation}
	\begin{equation}
	\label{eq:string}
	(\rho_i \rho_j)^2 = 1, \textrm{ for $i,j \in \{0,\ldots,n-1\}$ with $|i-j| \geq 2$}.
	\end{equation}
	Such a group is called a \emph{string group generated by involutions} (sggi).
	Then $\G$ is a string C-group if it also satisfies the following \emph{intersection condition} 
	for all subsets $I$ and $J$ of $\{0, \ldots, n-1\}$:
	\begin{equation}
	\label{eq:int-cond}
	\langle \rho_i \mid i \in I \rangle \cap \langle \rho_i \mid i \in J \rangle = \langle \rho_i \mid i \in I \cap J \rangle.
	\end{equation}
	
	Regular polytopes are always equivelar. If $\calP$ is a regular polytope of type $\{p_1, \ldots, p_{n-1}\}$,
	then $\G(\calP)$ is a quotient of the string Coxeter group
	\begin{align}
	\begin{split}
	[p_1, \ldots, p_{n-1}] := \langle \rho_0, \ldots, \rho_{n-1} \mid & \rho_i^2 = 1 \textrm{ for $0 \leq i \leq n-1$}, \\
		& (\rho_{i-1} \rho_i)^{p_i} = 1 \textrm{ for $1 \leq i \leq n-1$}, \\
		& (\rho_i \rho_j)^2 = 1 \textrm{ for $i,j \in \{0,\ldots,n-1\}$ with $|i-j| \geq 2$} \rangle.
	\end{split}
	\end{align}
	The facets of a regular $n$-polytope are all isomorphic to some regular polytope $\calK$, and the vertex-figures
	are isomorphic to some regular polytope $\calL$.

	A polytope is \emph{chiral} if the flags fall into
	two orbits with the property that adjacent flags lie in different orbits.
	Basic information about chiral polytopes can be found in \cite{chiral}, and a survey of important problems
	can be found in \cite{chiral-problems}. The facets of a chiral polytope
	are all isomorphic, as are the vertex-figures. Both the facets and the vertex-figures are either
	chiral or regular. Furthermore, the facets of the facets and the vertex-figures of the vertex-figures
	must be regular.

	If $\calP$ is a chiral polytope of type $\{\calK, \calL\}$, with $\calK$ and $\calL$ regular,
	then there is a unique minimal regular polytope $\calR$ that covers $\calP$. The polytope $\calR$ is
	called the \emph{mixed regular cover} of $\calP$, and is also of type $\{\calK, \calL\}$ (see \cite[Sec. 4]{high-rank-chiral}).
	
	\comment{
		This result can be generalized as follows.
		
		\begin{proposition}
		\label{regular-cover}
		Suppose $\calP$ is a polytope of type $\{\calK, \calL\}$, where $\calK$ and $\calL$ are regular.
		Then there is a unique minimal regular polytope $\calR$ of type $\{\calK, \calL\}$ such that $\calR$ covers $\calP$.
		\end{proposition}
		
		\begin{proof}
		This follows from Theorem 5.5 and Corollary 6.5 in \cite{mixing-and-monodromy}.
		\end{proof}
	}

	\subsection{Degenerate and flat polytopes}
	
		A polytope of type $\{p_1, \ldots, p_{n-1}\}$ is said to be \emph{degenerate} if at least
		one of the numbers $p_i$ is 2. A polytope is called \emph{flat} if every facet is incident with every vertex. 
		More generally, if $0 \leq k < m \leq n-1$, then an $n$-polytope is $(k,m)$-flat if every $k$-face is incident with
		every $m$-face. We summarize some properties of flatness below (see \cite[Prop. 2B16, Section 4E]{arp}).
		
		\begin{proposition}
		\label{flatness}
		Let $\calP$ be an $n$-polytope.
		\begin{enumerate}
		\item If $\calP$ is degenerate, then it is flat.
		\item If $0 \leq i \leq k < m \leq j \leq n-1$ and $\calP$ is $(k,m)$-flat,
		then $\calP$ is also $(i, j)$-flat.
		\item If $m \leq n-2$, then $\calP$ is $(k,m)$-flat if and only if the facets of $\calP$ are $(k,m)$-flat.
		\item If $k \geq 1$, then $\calP$ is $(k,m)$-flat if and only if the vertex-figures of $\calP$ are $(k-1,m-1)$-flat.
		\end{enumerate}
		\end{proposition}
		
		A polytope of type $\{p_1, \ldots, p_{n-1}\}$ is called \emph{tight} if it has exactly
		$2p_1 \cdots p_{n-1}$ flags, which is the minimum possible for a polytope of that type. 
		Tightness and flatness are related by the following result.
		
		\begin{proposition}[{\cite[Theorem 4.4]{tight-polytopes}}]
		\label{flat-tight}
		For $n \geq 2$, an equivelar $n$-polytope is tight if and only if it is $(i, i+2)$-flat for
		every $i$ satisfying $0 \leq i \leq n-3$.
		\end{proposition}

\section{Restrictions on chiral polytopes}
\label{restrictions}

	The study of chiral polytopes is, in many ways, still in its infancy. A number of general
	methods for constructing chiral polytopes have been discovered (see \cite{chiral-mix,
	abelian-chiral-covers, chiral-ext2, high-rank-chiral}), but
	few structural results are known. Perhaps the most fundamental question is: which
	regular polytopes can occur as the facets of a chiral polytope? We start with a simple result.
	
	\begin{theorem}
	\label{chiral-no-flat}
	There are no chiral polytopes with flat, finite, regular facets and finite regular vertex-figures.
	\end{theorem}
	
	\begin{proof}
	Suppose that $\calP$ is a chiral polytope of type $\{\calK, \calL\}$, where $\calK$ and $\calL$ are finite
	regular polytopes, and $\calK$ is flat. The mixed regular cover of $\calP$ is a regular polytope $\calQ$
	of type $\{\calK, \calL\}$. Now, since $\calK$ is flat, so are $\calP$ and $\calQ$, by \pref{flatness}(c).
	This means that $\calP$ and $\calQ$ both have the same number of vertices; namely, the number
	of vertices that $\calK$ has. Then since $\calP$ and $\calQ$ have isomorphic vertex-figures,
	and $\calQ$ covers $\calP$, it follows that $\calQ \cong \calP$, which is impossible since
	$\calP$ is chiral and $\calQ$ is regular.
	\end{proof}

	\tref{chiral-no-flat} leads to several further restrictions on the structure of chiral polytopes.

	\begin{theorem}
	\label{forbidden-facets}
	If $\calK$ is a regular $n$-polytope that is $(1, n-1)$-flat, then no finite chiral $(n+1)$-polytope
	has $\calK$ as a facet.
	\end{theorem}
	
	\begin{proof}
	Let $\calP$ be a finite chiral $(n+1)$-polytope of type $\{\calK, \calL\}$, and suppose that $\calK$ is a (finite) regular
	polytope that is $(1, n-1)$-flat. Then by \tref{chiral-no-flat}, the vertex-figures $\calL$ of $\calP$ must
	be chiral. Now, the facets of $\calL$ are isomorphic to the vertex-figures of $\calK$, which
	by \pref{flatness}(d) must be isomorphic to a regular, $(0, n-2)$-flat polytope of rank $n-1$. Then \tref{chiral-no-flat}
	implies that the vertex-figures of $\calL$ must be chiral. But this is impossible, since the
	vertex-figures of the vertex-figures of a chiral polytope are always regular.
	\end{proof}
	
	For example, let $\calP$ be the universal polytope of type $\{\{4, 3\}, \{3, 6\}_{(1,1)}\}$
	(denoted by $\{4, 3, 6\}*288$ in \cite{atlas}).
	Then the vertex-figures of $\calP$ are $(0,2)$-flat, and thus $\calP$ itself is $(1,3)$-flat.
	By \tref{forbidden-facets}, no finite chiral polytope has $\calP$ as a facet.
	Note that this gives a negative answer to Problem 28 in \cite{chiral-problems}.

	As a consequence of \tref{forbidden-facets}, we find that finite chiral polytopes cannot be arbitrarily flat.
	
	\begin{corollary}
	\label{flatness-restriction}
	There are no finite chiral $n$-polytopes that are $(1, n-3)$-flat or $(2, n-2)$-flat.
	\end{corollary}
	
	\begin{proof}
	Suppose $\calP$ is a finite chiral $n$-polytope that is $(1, n-3)$-flat. Then the facets of $\calP$
	are $(n-1)$-polytopes that are also $(1, n-3)$-flat, and \tref{forbidden-facets} implies
	that these facets cannot be regular. So the facets of $\calP$ are isomorphic to a finite chiral $(n-1)$-polytope
	$\calQ$ that is $(1, n-3)$-flat. But then $\calQ$ itself must have regular facets, and those facets
	are $(n-2)$-polytopes that are $(1, n-3)$-flat, contradicting \tref{forbidden-facets}.
	
	The second half follows since the dual of a $(2, n-2)$-flat $n$-polytope is $(1, n-3)$-flat.
	\end{proof}
	
	By \pref{flat-tight}, a tight polytope must be $(1, 3)$-flat. Thus, \cref{flatness-restriction}
	implies the following.
	
	\begin{corollary}
	\label{no-tight}
	There are no tight chiral $n$-polytopes with $n \geq 6$.
	\end{corollary}
	
	The Schl\"afli symbols of tight chiral polyhedra were classified in \cite{tight-chiral-polyhedra}. 
	Tight chiral 4-polytopes and 5-polytopes are further restricted due to \tref{chiral-no-flat}.
	
	\begin{theorem}
	If $\calP$ is a tight chiral $4$-polytope, then it has chiral facets or chiral vertex-figures (or both).
	If $\calP$ is a tight chiral $5$-polytope, then it has chiral facets, vertex-figures, and medial sections.
	\end{theorem}
	
	\begin{proof}
	Suppose that $\calP$ is a tight chiral $4$-polytope. Then the facets and vertex-figures
	of $\calP$ are both tight, and thus flat. Then \tref{chiral-no-flat} implies that the facets and
	vertex-figures cannot both be regular, so at least one of them is chiral. If instead $\calP$ is
	a tight chiral $5$-polytope, then the same result says that either the facets
	or the vertex-figures are tight chiral 4-polytopes. In either case, since the facets of the facets
	and the vertex-figures of the vertex-figures of $\calP$ must both be regular, the medial sections
	of $\calP$ must be chiral, which forces the facets and vertex-figures to both be chiral.
	\end{proof}

	The list of chiral polytopes at \cite{chiral-atlas} includes many tight chiral 4-polytopes.
	So far, no tight chiral 5-polytopes have been discovered. The obvious candidates, with facets
	and vertex-figures isomorphic to tight chiral 4-polytopes, seem to always collapse to
	something regular or something non-polytopal.

	\begin{problem}
	Fully classify the tight chiral polyhedra and $4$-polytopes.
	\end{problem}
	
	\begin{problem}
	Determine whether there are any tight chiral 5-polytopes.
	\end{problem}

	Our next goal will be to determine a lower bound for the number of flags of a chiral $n$-polytope.
	To do so, we will need to determine the smallest non-flat regular polytopes in each rank.
	
\section{Non-flat regular polytopes}
\label{non-flat-size}

	Recall that a polytope is flat if every vertex is incident on every facet. Thus,
	if a polytope is not flat, then it has at least one more vertex than its facets have.
	This yields the following simple consequences.
	
	\begin{proposition}
	\label{non-flat-verts}
	If $\calP$ is a regular non-flat $n$-polytope of type $\{p_1, \ldots, p_{n-1}\}$, with $n \geq 3$, then $\calP$ has	at least $p_{n-1} + n - 2$ facets and at least $p_1 + n - 2$ vertices.
	\end{proposition}
	
	\begin{proof}
	If $n = 3$, then the facets are $p_1$-gons, so in order for $\calP$ to be non-flat, it must have at least
	$p_1 + 1$ vertices. Similarly, the vertex-figures are $p_2$-gons, so $\calP$ must have at least $p_2 +1$
	facets in order to be non-flat. The claim then follows by induction on $n$.
	\end{proof}

	\begin{corollary}
	\label{non-flat-verts2}
	A non-flat regular $n$-polytope has at least $n+1$ facets and $n+1$ vertices.
	\end{corollary}
	
	\begin{proof}
	In light of \pref{non-flat-verts}, the only way to have fewer than $n+1$ facets or vertices
	is for $p_1$ or $p_{n-1}$ to be 2. But then $\calP$ is flat, by \pref{flatness}(a).
	\end{proof}

	In fact, the fewer vertices that a polytope has (in a fixed rank), the flatter it must be.
	
	\begin{proposition}	
	\label{few-vertices-flat}
	Suppose $\calP$ is a regular $n$-polytope of type $\{p_1, \ldots, p_{n-1}\}$ with $k$ vertices, 
	$k \leq p_1 + n - 3$. Then $\calP$ is $(0, k+2-p_1)$-flat.
	\end{proposition}	
	
	\begin{proof}
	First, suppose that $k = p_1 + n - 3$. Then \pref{non-flat-verts2} implies that $\calP$ is flat,
	i.e., $(0, n-1)$-flat, as desired. For the case $n = 3$, we are done, since $\calP$ has at least 
	$p_1$ vertices.
	
	Now suppose that the claim is true for $(n-1)$-polytopes with $k' \leq p_1 + (n-1) - 3$,
	and suppose that $\calP$ has $k < p_1 + n - 3$ vertices. Then the facets have $k' \leq k < p_1 + n - 3$ vertices.
	Therefore, $k' \leq p_1 + (n-1) - 3$, and by inductive hypothesis, the facets are $(0, k'+2-p_1)$-flat.
	Then \pref{flatness}(c) shows that $\calP$ is $(0, k'+2-p_1)$-flat. Since
	$k' \leq k$, this implies that $\calP$ is $(0, k+2-p_1)$-flat, by \pref{flatness}(b).
	\end{proof}

	\begin{corollary}	
	\label{few-vertices-flat2}
	Suppose $\calP$ is a regular $n$-polytope with $k$ vertices, $k \leq n$.
	Then $\calP$ is $(0, k-1)$-flat.
	\end{corollary}
	
	\begin{proof}
	If $p_1 = 2$, then \cite[Prop. 2B16]{arp} says that $\calP$ is $(0,1)$-flat, 
	which implies that it is $(0, k-1)$-flat.
	Otherwise, if $p_1 \geq 3$, then having $k \leq n$ implies that $k \leq p_1 + n - 3$,
	and so \pref{few-vertices-flat} implies that $\calP$ is $(0, k+2-p_1)$-flat.
	Since $k+2-p_1 \leq k-1$, this implies that $\calP$ is $(0, k-1)$-flat, by \pref{flatness}(b).
	\end{proof}

	The preceding results are already enough to determine the smallest non-flat regular polytopes.

	\begin{proposition}
	\label{simplex-smallest}
	The simplex is the unique smallest non-flat regular $n$-polytope.
	\end{proposition}
	
	\begin{proof}
	We use induction on $n$. The claim is clearly true for $n = 2$. In general, if $\calP$ is a non-flat
	regular $n$-polytope, then its facets are non-flat regular $(n-1)$-polytopes. By inductive hypothesis,
	the facets each are at least as large as simplices, with $n!$ flags. Then since \pref{non-flat-verts}
	implies that $\calP$ has at least $n+1$ facets, it follows that $\calP$ has at least $(n+1)!$ flags.
	Furthermore, the only way for $\calP$ to have exactly $(n+1)!$ flags is if it has $n+1$ facets that all
	have $n!$ flags. By inductive hypothesis, the facets must be simplices, and if there are $n+1$ facets,
	then $p_{n-1} = 3$ by \pref{non-flat-verts}. So $\calP$ must be a simplex.
	\end{proof}
	
	Before we continue to find small non-flat regular polytopes, let us describe
	a family of regular polytopes, which we will call \emph{central extensions
	of simplices}. 	Consider a sequence $p_1, \ldots, p_{n-1}$, where each $p_i$ is either 3 or 6.
	Let $\Lambda(p_1, \ldots, p_{n-1})$ be the quotient of $[p_1, \ldots, p_{n-1}]$
	by the relations that make each $(\rho_{i-1} \rho_i)^3$ central.

	\begin{proposition}
	$\Lambda(p_1, \ldots, p_{n-1})$ is the automorphism group of a regular $n$-polytope of type
	$\{p_1, \ldots, p_{n-1}\}$ and with $\ds \frac{p_1 \cdots p_{n-1}}{3^{n-1}} (n+1)!$ flags.
	\end{proposition}
	
	\begin{proof}
	We start by verifying that the order of each $\rho_{i-1} \rho_i$ is $p_i$. Let $\Lambda = \Lambda(p_1, \ldots,
	p_{n-1})$. Clearly $\Lambda$ covers $[3, \ldots, 3]$, and so the order of each $\rho_{i-1} \rho_i$ is
	divisible by 3. Now, consider one $p_k$ such that $p_k = 6$. It is straightforward to verify that
	there is an epimorphism $\pi: \Lambda \to \langle x, y \mid x^2 = y^2 = (xy)^2 = 1 \rangle$
	such that
	\[ \rho_j \pi = \begin{cases}
		x, \textrm{ if $j \leq k-1$} \\
		y, \textrm{ if $j \geq k$} \\
	\end{cases} \]
	It follows that whenever $p_k = 6$, the order of $\rho_{k-1} \rho_k$ is divisible by 2.
	It's clear then that the order of each $\rho_{i-1} \rho_i$ is $p_i$.

	The subgroup
	\[ N = \langle (\rho_0 \rho_1)^3, \ldots, (\rho_{n-2} \rho_{n-1})^3 \rangle \]
	is central in $\Lambda$, and has order $p_1 \cdots p_{n-1} / 3^{n-1}$.
	Furthermore, the quotient of $\Lambda$ by $N$ is the group of the $n$-simplex $[3, \ldots, 3]$,
	of order $(n+1)!$. That proves that $\Lambda$ has the desired order.

	It remains to prove that $\Lambda$ is a string C-group. First, note that
	$\Lambda(6, p_2, \ldots, p_{n-1})$ covers $\Lambda(3, p_2, \ldots, p_{n-1})$, and this cover is one-to-one
	on the subgroup $\langle \rho_1, \ldots, \rho_{n-1} \rangle$. Then the quotient criterion 
	(see \cite[Thm. 2E17]{arp}) implies
	that the former is a string C-group provided that the latter is. The same argument works with
	$\Lambda(p_1, \ldots, p_{n-2}, 6)$. So to prove the result, it suffices to prove it 
	for $\Lambda(3, p_2, \ldots, p_{n-2}, 3)$. This already settles the case $n = 3$.
	
	Suppose now that $n \geq 4$, that $p_1 = p_{n-1} = 3$, and that the subgroups 
	$\langle \rho_0, \ldots, \rho_{n-2} \rangle$ and
	$\langle \rho_1, \ldots, \rho_{n-1} \rangle$ are both string C-groups.
	Let $\varphi \in \langle \rho_0, \ldots, \rho_{n-2} \rangle \cap \langle \rho_1, \ldots, \rho_{n-1} \rangle$.
	To prove that $\Lambda$ itself is a string C-group, it suffices to show that
	$\varphi \in \langle \rho_1, \ldots, \rho_{n-2} \rangle$ (by \cite[Prop. 2E16(a)]{arp}).
	Let $\pi: \Lambda \to \Lambda / N \cong [3, \ldots, 3]$, and note that $N \leq \langle \rho_1, \ldots, \rho_{n-2} \rangle$
	since $p_1 = p_{n-1} = 3$. Denoting the image of $\rho_i$ under $\pi$ by $\overline{\rho_i}$,
	we have that $\overline{\varphi}$ lies in
	$\langle \overline{\rho_0}, \ldots, \overline{\rho_{n-2}} \rangle \cap
	\langle \overline{\rho_1}, \ldots, \overline{\rho_{n-1}} \rangle$.
	Since $\G/N = [3, \ldots, 3]$ is a string C-group, it follows that $\overline{\varphi} \in \langle \overline{\rho_1}, \ldots,
	\overline{\rho_{n-2}} \rangle$, and thus $\varphi \in \langle \rho_1, \ldots, \rho_{n-2} \rangle N = \langle
	\rho_1, \ldots, \rho_{n-2} \rangle$.
	Thus, $\Lambda$ is a string C-group provided that its facet subgroup and vertex-figure subgroup are
	string C-groups, and the result follows by induction on the rank of $\Lambda$.
	\end{proof}

	Let $\calP(p_1, \ldots, p_{n-1})$ be the polytope whose automorphism group is
	$\Lambda(p_1, \ldots, p_{n-1})$. We call this polytope a central extension of a simplex.
	The group of the vertex-figure is $\Lambda(p_2, \ldots, p_{n-1})$, which has index $(n+1)p_1/3$
	in $\Lambda(p_1, \ldots, p_{n-1})$. Thus, the polytope $\calP(p_1, \ldots, p_{n-1})$ has
	$(n+1)p_1/3$ vertices, while its facets $\calP(p_1, \ldots, p_{n-2})$ have $np_1/3$ vertices.
	This shows that these polytopes are not flat.
	
	Next, let us show that any polytope built out of central extensions of simplices is
	itself a central extension of a simplex.

	\begin{proposition}
	\label{simplex-amalgamation}
	If $n \geq 4$ and $\calP$ is an $n$-polytope of type 
	$\{\calP(p_1, \ldots, p_{n-2}), \calP(p_2, \ldots, p_{n-1})\}$,
	then $\calP \cong \calP(p_1, \ldots, p_{n-1})$. 
	\end{proposition}
	
	\begin{proof}
	Clearly $\calP$ is a quotient of $\calP(p_1, \ldots, p_{n-1})$. The facets $\calK$ of $\calP$ have $np_1/3$
	vertices, and so $\calP$ itself has at least $np_1/3$ vertices. Since $\calP$ is covered by
	$\calP(p_1, \ldots, p_{n-1})$, which has $(n+1)p_1/3$ vertices, the number of vertices of $\calP$ must
	divide $(n+1)p_1/3$.  It follows that $\calP$ itself has $(n+1)p_1/3$ vertices and that 
	$\calP \cong \calP(p_1, \ldots, p_{n-1})$.
	\end{proof}
	
	\comment{
		\begin{proposition}
		\label{non-simplex}
		Suppose that the smallest non-flat regular $r$-polytope that is not a simplex or a central extension of a simplex 
		has $N$ flags.
		Then for all $n \geq r$, the smallest non-flat regular $n$-polytope that is not a simplex or a central extension of a simplex
		has at least $N \cdot (n+1)!/(r+1)!$ flags.
		\end{proposition}
		
		\begin{proof}
		We prove the claim using induction on $n$. The claim is trivially true when $n = r$.
		Now, suppose that $n > r$ and that $\calP$ is a non-flat regular $n$-polytope that is not a simplex or a central 
		extension of a simplex. Then by \pref{simplex-amalgamation}, the facets or vertex-figures of $\calP$ 
		are not simplices or central extensions of simplices. By inductive hypothesis, the facets or vertex-figures
		have at least $N \cdot n!/(r+1)!$ flags, and by \pref{non-flat-verts2}, $\calP$ has at least $n+1$ facets and
		$n+1$ vertices. The claim follows.
		\end{proof}
	}
	
	Our goal now is to find the several smallest non-flat regular polytopes in each rank.
	
	\begin{proposition}
	\label{second-smallest}
	Suppose $\calP$ is the second smallest non-flat regular $n$-polytope, and $n \geq 3$. Then
	$\calP$ has $2(n+1)!$ flags. Furthermore, if $n \geq 4$, then $\calP$ is a central extension
	of a simplex.
	\end{proposition}
	
	\begin{proof}
	We prove the claim by induction on $n$. The claim can be shown to be true for $n = 3$ and $n=4$, using
	\cite{atlas}. Suppose that $n \geq 5$. Since $\calP$ is the second smallest non-flat regular
	$n$-polytope, it is not a simplex. Then up to duality, we may assume that the facets of $\calP$ 
	are not simplices. These facets have at least $2n!$ flags, by inductive hypothesis, and there are at 
	least $n+1$ of them (by \cref{non-flat-verts2}), so $\calP$ has at least $2(n+1)!$ flags.
	On the other hand, the polytopes $\calP(p_1, \ldots, p_{n-1})$ with a single $p_i = 6$ have exactly
	$2(n+1)!$ flags, and so if $\calP$ is the second smallest, it must have exactly $2(n+1)!$ flags.
	It follows that the facets have exactly $2n!$ flags, and by inductive hypothesis, these facets
	are central extensions of simplices. Similarly, the vertex-figures cannot have more than
	$2n!$ flags, since there are at least $n+1$ vertices, so the vertex-figures are either
	simplices or central extensions of simplices. \pref{simplex-amalgamation} then implies that
	$\calP$ is itself a central extension of a simplex.
	\end{proof}

	\begin{proposition}
	\label{third-smallest}
	Suppose $\calP$ is the third smallest non-flat regular $n$-polytope.
	\begin{enumerate}
	\item If $n = 3$, then $\calP$ has $60$ flags.
	\item If $n = 4$, then $\calP$ has $384$ flags.
	\item If $n \geq 5$, then $\calP$ is a central extension of a simplex, with $4(n+1)!$ flags.
	\end{enumerate}
	\end{proposition}

	\begin{proof}
	For $n = 3, 4$, and $5$, we may verify the claim directly using \cite{conder-atlas}. 
	For $n \geq 6$, the proof is essentially the same as the proof of \pref{second-smallest}.
	\end{proof}
	
	\begin{proposition}
	\label{fourth-smallest}
	Suppose $\calP$ is the fourth smallest non-flat regular $n$-polytope, with $n \geq 5$.
	Then $\calP$ has at least $\ds (16/3)(n+1)!$ flags.
	\end{proposition}
	
	\begin{proof}
	The claim can be verified for $n=5$ using \cite{conder-atlas}. 
	Now suppose that $n \geq 6$. Up to duality, we may assume that the facets of $\calP$ have at least
	as many flags as the vertex-figures. If $\calP$ is a central extension of a simplex with more
	than $4(n+1)!$ flags, then it has at least $8(n+1)!$ flags. Otherwise, if $\calP$ is not a central
	extension of a simplex, then its facets must have at least $(16/3)n!$ flags, and there are at least
	$n+1$ facets, so $\calP$ has at least $(16/3)(n+1)!$ flags.
	\end{proof}

	Of course, there is no particular reason to stop at the fourth smallest polytopes --- except that we have
	reached the limits of the data we have on small regular polytopes, which has established the base cases
	in the previous several results. Table~\ref{smallest-data} summarizes our results.
	
	\begin{table}[h!]
	\caption{Number of flags of the smallest non-flat regular polytopes}
	\begin{center}
	\begin{tabular}{ccccc}
	\hline
	Rank & smallest & \shortstack{second \\ smallest} & \shortstack{third \\ smallest} & \shortstack{fourth \\ smallest} \\ \hline
	3 & 24 & 48 & 60 & 64 \\
	4 & 120 & 240 & 384 & 480 \\
	5 & 720 & 1440 & 2880 & 3840 \\
	$n \geq 6$ & $(n+1)!$ & $2(n+1)!$ & $4(n+1)!$ & $\geq (16/3)(n+1)!$ \\ \hline
	\end{tabular}
	\end{center}
	\label{smallest-data}
	\end{table}

	A solution to the following problem would be a good step toward a fuller understanding of small
	non-flat regular polytopes.
	
	\begin{problem}
	In each rank, determine the smallest non-flat regular polytope that is not a central extension of a
	simplex.
	\end{problem}
	
	\comment{
		When $n = 3$, there are two non-flat regular polyhedra (up to duality) with 48 flags; the octahedron
		$\{3, 4\}$, and the toroidal polyhedron $\{3, 6\}_{(2,0)}$. When $n = 4$, there are again two non-flat
		regular polytopes (up to duality) with 240 flags; a polytope of type $\{3, 3, 6\}$ where the vertex-figures
		are $\{3, 6\}_{(2,0)}$, and one of type $\{3, 6, 3\}$ where the facets are $\{3, 6\}_{(2,0)}$ and the
		vertex-figures are $\{6, 3\}_{(2,0)}$. For $n \geq 5$, we note that if $\calP$ has exactly
		$2(n+1)!$ flags, then without loss of generality it has facets with at least $2n!$ flags. Since there are
		at least $n+1$ facets, we see that $\calP$ must have exactly $n+1$ facets with exactly $2n!$ flags.
		Furthermore, the vertex-figures must either be simplices or polytopes with $2n!$ flags.
		This leads to the following classification.
	}

\section{Small chiral polytopes}

	The restrictions in the previous section help us describe general
	lower bounds on the size of chiral polytopes. We will need the following result.
	
	\begin{proposition}
	\label{3-verts}
	Chiral polytopes have at least $3$ vertices and at least $3$ facets.
	\end{proposition}
	
	\begin{proof}
	If $\calP$ is a polytope with 2 vertices, then every edge is incident on both vertices.
	Thus, the two vertices are indistinguishable, and there is an automorphism of $\calP$
	that swaps the vertices while fixing all other faces. This yields two adjacent flags
	that lie in the same orbit, and so $\calP$ is not chiral. The proof of the other claim is essentially the same.
	\end{proof}
	
	\comment{
		\begin{theorem}
		\label{small-nonregular}
		{\bf REWRITE TO ONLY BE ABOUT CHIRALS}
		Suppose that $n \geq 6$ and that $\calP$ is a non-regular $n$-polytope of type $\{\calK, \calL\}$,
		where $\calK$ and $\calL$ are regular. Then $\calP$ has at least $(16/3)n \cdot n!$ flags.
		\end{theorem}
		
		\begin{proof}
		By \tref{chiral-no-flat}, both $\calK$ and $\calL$ must be non-flat since $\calP$ is chiral.
		Furthermore, either $\calK$ or $\calL$ must not be a central extension of a simplex, because otherwise
		\pref{simplex-amalgamation} would imply that $\calP$ is also a central extension of a simplex,
		which is regular. So $\calK$ (say) is not a central extension of a simplex, and thus it has
		at least $(16/3)n!$ flags. Since $\calL$ is not flat, it has at least $n$ facets, and so
		$\calP$ itself also has at least $n$ facets. Thus $\calP$ has at least $(16/3)n \cdot n!$ flags.
		\end{proof}
	}

	Now we can provide lower bounds on the size of a chiral polytope, depending on whether the facets
	and vertex-figures are regular or chiral.
	
	\begin{theorem}
	\label{smallest-chiral}
	Let $\calP$ be a chiral $n$-polytope with regular facets and vertex-figures.
	If $n = 5$, then $\calP$ has at least 4004 flags.
	If $n \geq 6$ then $\calP$ has at least $(16/3) n \cdot n!$ flags.
	\end{theorem}
	
	\begin{proof}
	There are only three chiral $5$-polytopes with at most 4000 flags, and all
	have chiral facets and vertex-figures (see \cite{chiral-atlas}). 
	Since the number of flags of a polytope is always divisible by 4,
	$\calP$ must have at least $4004$ flags in this case. 
	
	Now suppose that $n \geq 6$, and that $\calP$ is of type $\{\calK, \calL\}$, with $\calK$ and
	$\calL$ regular.
	By \tref{chiral-no-flat}, both $\calK$ and $\calL$ must be non-flat since $\calP$ is chiral.
	Furthermore, either $\calK$ or $\calL$ must not be a central extension of a simplex, because otherwise
	\pref{simplex-amalgamation} would imply that $\calP$ is also a central extension of a simplex,
	which is regular. Up to duality, we may assume that $\calK$ is not a central extension of a simplex, and thus it has
	at least $(16/3)n!$ flags. Since $\calL$ is not flat, it has at least $n$ facets, and so
	$\calP$ itself also has at least $n$ facets. Thus $\calP$ has at least $(16/3)n \cdot n!$ flags.
	\end{proof}

	\begin{theorem}
	\label{smallest-chiral2}
	Let $\calP$ be a chiral $n$-polytope with chiral facets and regular vertex-figures.
	If $n = 6$, then $\calP$ has at least $18432$ flags.
	If $n \geq 7$ then $\calP$ has at least $16(n-1)(n-1)!$ flags.
	\end{theorem}
	
	\begin{proof}
	For $n=6$, the result follows from \cite[Thm 1.1]{smallest-chiral-6}, which proves that the smallest
	chiral 6-polytope has 18432 flags.
	Let $\calK$ be the facet type of $\calP$. Since the facets of the facets of a chiral polytope
	must be regular, $\calK$ has regular facets. Furthermore, since the vertex-figures of $\calK$ are
	also the facets of the regular vertex-figures of $\calP$, $\calK$ has regular
	vertex-figures. So $\calK$ is a chiral $(n-1)$-polytope with regular facets and
	vertex-figures. By \pref{3-verts}, $\calP$ must have at least 3 facets, 
	and combining this with \tref{smallest-chiral}
	yields the desired result.
	\end{proof}
	
	\begin{theorem}
	\label{smallest-chiral3}
	Let $\calP$ be a chiral $n$-polytope with chiral facets and chiral vertex-figures.
	If $n = 7$, then $\calP$ has at least $55296$ flags.
	If $n \geq 8$ then $\calP$ has at least $48(n-2)(n-2)!$ flags.
	\end{theorem}
	
	\begin{proof}
	The proof is essentially the same as for \tref{smallest-chiral2}; the facets of $\calP$ must be chiral with
	regular facets and either regular or chiral vertex-figures, and there are at least 3 facets.
	Applying \tref{smallest-chiral} and (the dual of) \tref{smallest-chiral2} for the facets provides the desired result.
	\end{proof}
	
	We note that if $n \geq 8$, then $48(n-2)(n-2)! < 16(n-1)(n-1)! < (16/3)n\cdot n!$, providing us with the
	following theorem.
	
	\begin{theorem}
	\label{smallest-chiral-n}
	For $n \geq 8$, a chiral $n$-polytope has at least $48(n-2)(n-2)!$ flags.
	\end{theorem}
	
	Compare \tref{smallest-chiral-n} to \cite[Thm. 1.1]{smallest-regular}, which states that for $n \geq 9$, the smallest
	nondegenerate regular $n$-polytope has $2 \cdot 4^{n-1}$ flags. Note that these regular polytopes are all flat,
	and in fact, they all have flat regular facets and flat regular vertex-figures.

	\subsection{The smallest chiral polytopes in each rank}
	\label{smallest-chiral-polys}
	
		The smallest chiral $n$-polytopes for $n = 3, 4,$ and $5$ can be found in \cite{chiral-constructions}.
		In rank 3, the smallest chiral polytope is the torus map $\{4, 4\}_{(1, 2)}$, with 40 flags. 
		In rank 4, the smallest chiral polytopes have 240 flags. This includes the universal
		$\{\{4, 4\}_{(1,2)}, \{4, 3\}\}$ and the universal $\{\{4, 4\}_{(2,1)}, \{4, 4\}_{(1,2)}\}$;
		the former has chiral facets and regular vertex-figures, and the latter has chiral facets and chiral
		vertex-figures.
		The smallest chiral 5-polytope is the universal $\{ \{\{3, 4\}, \{4, 4\}_{(2,1)}\}, \{\{4, 4\}_{(1,2)}, \{4, 3\}\}\}$,
		with 1440 flags. This polytope has chiral facets and vertex-figures.

		The smallest chiral 4-polytope with regular facets and vertex-figures is a polytope of type
		$\{3, 3, 8\}$ with 384 flags. It has automorphism group
		\[ \langle \s_1, \s_2, \s_3 \mid \s_1^3 = \s_2^3 = \s_3^8 = (\s_1 \s_2)^2 = (\s_2 \s_3)^2 = (\s_1 \s_2 \s_3)^2
		= \s_3^{-1} \s_1 \s_3 \s_2^{-1} \s_1 \s_3^{-2} \s_2 = 1 \rangle. \]
		
		There are no chiral 5-polytopes with up to
		4000 flags and regular facets or vertex-figures (see \cite{chiral-atlas}). 
		The smallest known chiral 5-polytopes with either
		regular facets or regular vertex-figures are described in \cite{smallest-chiral-6}.
		One is of type $\{3, 3, 4, 6\}$, with regular facets and chiral vertex-figures, and the
		other is of type $\{3, 4, 6, 3\}$, with chiral facets and regular vertex-figures. Both have 4608 flags.
		The smallest chiral 6-polytope has the former 5-polytope as facets, the latter as vertex-figures,
		and has 18432 flags. We summarize our data in Table~\ref{chiral-data}.

		\begin{table}[h]
		\caption{Number of flags of the smallest chiral polytopes}
		\begin{center}
		\begin{tabular}{c c c c}
		\hline
		Rank & \shortstack{regular facets \\ regular vertex-figures} & \shortstack{chiral facets \\ regular vertex-figures} & 
		\shortstack{chiral facets \\ chiral vertex-figures} \\ \hline
		3 & 40 & - & - \\
		4 & 384 & 240 & 240 \\
		5 & $\geq 4004$ & $\geq 4004, \leq 4608$ & 1440 \\
		6 & $\geq 23040$ & $\geq 18432$ & 18432 \\
		7 & $\geq 188160$ & $\geq 69120$ & $\geq 55296$ \\
		$n \geq 8$ & $\geq (16/3)n \cdot n!$ & $\geq 16(n-1)(n-1)!$ & $\geq 48(n-2)(n-2)!$ \\ \hline
		\end{tabular}
		\end{center}
		\label{chiral-data}
		\end{table}
		
		Let $f_{rr}(n)$, $f_{cr}(n)$, and $f_{cc}(n)$ be the minimal number of flags among chiral $n$-polytopes with 
		regular facets and vertex-figures, with chiral facets and regular vertex-figures, and with chiral facets
		and vertex-figures, respectively. It is straightforward to prove that $f_{cr}(n) \geq 3f_{rr}(n-1)$
		and that $f_{cc}(n) \geq 3f_{cr}(n-1)$. From the data available, it seems likely that
		$f_{rr}(n) \geq f_{cr}(n) \geq f_{cc}(n)$, but it is unclear whether this trend will continue to hold.

		Problem 35 in \cite{chiral-problems} asks to find the size of the smallest chiral $n$-polytope for
		each $n$. It may be useful to split that problem into the following subproblems:
		
		\begin{problem}
		Determine the functions $f_{rr}(n)$, $f_{cr}(n)$, and $f_{cc}(n)$.
		\end{problem}
		
		\begin{problem}
		Determine whether it is always the case that $f_{rr}(n) \geq f_{cr}(n) \geq f_{cc}(n)$.
		\end{problem}
		
		The next smallest step in this direction would be to solve the following problem:
		
		\begin{problem}
		Determine the smallest chiral 5-polytope with regular facets and regular vertex-figures.
		\end{problem}
		
\section{Acknowledgment}

	The calculations in this paper were made possible using GAP \cite{gap}.
		
\bibliographystyle{amsplain}
\bibliography{gabe}

\end{document}